\documentclass[12pt]{article}

\usepackage{amsmath}
\usepackage{amssymb, latexsym, a4}
\usepackage{graphicx}
\usepackage[all]{xy}
\usepackage{multicol}
\usepackage[usenames]{color}
 
\usepackage{mathrsfs}
\usepackage{hyperref}
\usepackage{cleveref}
\newtheorem{theorem}{Theorem}[section]
\newtheorem{lemma}[theorem]{Lemma}
\newtheorem{remark}[theorem]{Remark}
\newtheorem{corollary}[theorem]{Corollary}
\newtheorem{proposition}[theorem]{Proposition}

\newtheorem{definition}[theorem]{Definition}

\newenvironment{proof}{\trivlist\item[]\rm{\textbf{Proof.}\ }}{\endtrivlist}
 \makeatletter
      \def\@setcopyright{}
      \def\serieslogo@{}
      \makeatother

\author{Hesam Safa  \\ }

\title{On multipliers of pairs of Lie superalgebras}

\begin{document}
\maketitle


\noindent\textbf{Abstract.} In this article, we  study the notion of the Schur multiplier $\mathcal{M}(N,L)$ of a pair $(N,L)$
of Lie superalgebras and obtain some upper bounds concerning dimensions.
 Moreover, we  characterize the pairs  of finite dimensional (nilpotent) Lie superalgebras for which
 $\dim \mathcal{M}(N,L)= \frac{1}{2}\big{(}(m+n)^2+(n-m)\big{)}+\dim N\dim(L/N)-t$,
 for $t=0,1$, where $\dim N=(m|n)$. \\


\section{Introduction}

The story of  Lie superalgebras begins  in the late 1960s. The Russian physicist Stavraki \cite{sta} is by far the first one who used the term
{\it supersymmetry}, at which time it attracted little attention. Then
 V. G. Kac \cite{kac} started working on this new algebraic structure and  published his first results on the classification of Lie superalgebras in 1971.
 The notion of Lie superalgebras also appeared almost simultaneously in a paper of Berezin and G. I. Kac (Kats) \cite{ber} in 1970.
Developing the properties of Lie superalgebras has been of some interest in mathematics and theoretical physics for the last 40 years
(see \cite{b-l,lad,gom,k} for more information).

In 1996, Batten et al. \cite{b-m-s} discussed and studied
the concept of the  Schur multiplier of a Lie algebra, which
is analogous to the Schur multiplier of a group  introduced by  Schur \cite{s} in 1904.  Moneyhun \cite{m} proved
for an $m$-dimensional Lie algebra $L$ that $\dim \mathcal{M}(L)\leq\frac{1}{2}m(m-1)$, where $\mathcal{M}(L)$ denotes  the Schur multiplier of
$L$. Also in \cite{s-s-e}, Saeedi et al. generalized the Moneyhun's result to a pair of Lie algebras and proved  that
if $(N,L)$ is a pair of Lie algebras in which $N$ admits a complement  in $L$ and $\dim N=m$, then
 $\dim \mathcal{M}(N,L)\leq\frac{1}{2}m\big{(}m+2\dim(L/N)-1\big{)}$.
Recently  in \cite{n1,miao}, the notion of the Schur multiplier has been extended  to Lie superalgebras. In \cite{n1}, Nayak shows that
if $L$ is a Lie superalgebra of dimension $(m|n)$, then
  $\dim \mathcal{M}(L)\leq \frac{1}{2}\big{(}(m+n)^2 + (n-m)\big{)}$.

  In the present paper, following \cite{n1,s-s-e,miao} we study the Schur multiplier  of   a pair of Lie superalgebras and provide an upper bound for its dimension, which simultaneously extends all above bounds.
  We also characterize the pairs $(N,L)$ of finite dimensional (nilpotent) Lie superalgebras whose Schur multipliers can reach a maximum dimension (Theorem \ref{th2})
  or one less than that (Theorem \ref{th5}). As a consequence, we show that
   the  Heisenberg Lie algebra of dimension 3 is the only $(m|n)$-dimensional nilpotent  Lie superalgebra $L$ for which
    $\dim \mathcal{M}(L)= \frac{1}{2}\big{(}(m+n)^2+(n-m)\big{)}-1$
   (see also \cite{miao}).


\section{Preliminaries}

Throughout this paper, all (super)algebras are considered over an algebraically closed  field $\mathbb{F}$ of characteristic $\not=2,3$.
We first discuss some  terminologies   on Lie superalgebras  from \cite{lad,k,n1}.

Let $\mathbb{Z}_2=\{0,1\}$ be a field. A $\mathbb{Z}_2$-graded vector space  (or superspace) $V$ is
a direct sum of vector spaces $V_0$ and $V_1$, whose elements are called even and odd, respectively.
Non-zero elements of $V_0\cup V_1$ are said to be homogeneous. For a homogeneous element $v\in V_a$ with
$a\in\mathbb{Z}_2$,  $|v|=a$ is the
degree of $v$. In the sequel,  when the notation $|v|$ appears, it means that $v$ is a
homogeneous element.
A vector subspace $U$ of $V$ is called $\mathbb{Z}_2$-graded vector subspace (or sub-superspace), if $U=U_0\oplus U_1$ where
$U_0=U\cap V_0$ and $U_1=U\cap V_1$.

\begin{definition}\label{def0}\normalfont
A  Lie superalgebra is a superspace $L=L_0\oplus L_1$ equipped with a bilinear mapping
$[-,-] : L \times L \to L$, usually called the graded bracket  of $L$, satisfying
the following conditions:
\begin{itemize}
\item[$(i)$] $[L_a,L_b]\subseteq L_{a+b}$, for every $a,b\in \mathbb{Z}_2$,
\item[$(ii)$] $[x,y]=-(-1)^{|x||y|} [y,x]$,
\item[$(iii)$] $(-1)^{|x||z|}[x,[y,z]]+ (-1)^{|y||x|}[y,[z,x]]+(-1)^{|z||y|}[z,[x,y]]=0$,
\end{itemize}
for every $x,y,z\in L$.
The identities $(ii)$ and $(iii)$ are called  graded  antisymmetric property and  graded Jacobi identity, respectively.
\end{definition}

  Clearly, $(i)$ means that $|[x,y]|= |x|+|y|$  (modulo 2).
  Moreover, since $\frac{1}{2},\frac{1}{3}\in \mathbb{F}$, $(ii)$ implies
that $[x,x]=0$ for all $x\in L_0$, and $(iii)$ implies that $[x,[x,x]]=0$ for all $x\in L$. Hence the even part $L_0$ of a Lie superalgebra $L$ is actually a Lie algebra,
   which means that if  $L_1=0$, then $L$ becomes a usual Lie algebra. Also, the odd part $L_1$ is an $L_0$-module.
 If $L_0=0$, then $[x,y]=0$, for all $x,y\in L$ and hence $L$ is an abelian Lie superalgebra.
Also,  one can easily see that the
graded Jacobi identity may
be replaced by
\[[[x,y],z]=[x,[y,z]]-(-1)^{|x||y|} [y,[x,z]].\]

A sub-superspace $I$ of a Lie superalgebra $L$ is said to be a  sub-superalgebra (resp.  graded ideal), if $[I,I]\subseteq I$
(resp. $[I,L]\subseteq I$).
Now, let $N$ be a graded ideal of a Lie superalgebra  $L$. Then $(N,L)$ is called a pair of Lie superalgebras.
The center and  commutator of  the pair $(N,L)$ are defined as
$Z(N,L)=\{n\in N|\ [n,x]=0, \forall x\in L\}$ and
$[N,L]=\langle [n,x]|\ n\in N,x\in L\rangle$, repectively,
which are  graded ideals of $L$, contained in $N$. Clearly if $N=L$, then the above graded ideals coincide with the usual
center and commutator of $L$.
 Also, a Lie superalgebra $L$ is said to be  nilpotent, if $L^{c}=0$, where $L^1=L$ and
  $L^{i+1}=[L^i,L]$, $i\geq 1$.

Let $L$ and $K$ be two Lie superalgebras. A linear map $f:L\to K$ is called a  homomorphism of Lie superalgebras, if
$f(L_a)\subseteq K_a$ for every $a\in\mathbb{Z}_2$, and $f([x,y])=[f(x),f(y)]$ for every $x,y\in L$ (see \cite{n1,safa0} for more details).
Throughout this paper when a Lie superalgebra $L=L_0\oplus L_1$ is of dimension $m+n$,  in which $\dim L_0=m$ and
$\dim L_1=n$, we write $\dim L=(m|n)$.

  In the context of Lie algebras,  the Schur multiplier of a pair $(N,L)$, denoted by $\mathcal{M}(N,L)$, appears in the following natural exact sequence of Lie algebras
 \begin{eqnarray*}
H_3(L)\rightarrow H_3(L/N)\rightarrow \mathcal{M}(N,L)\rightarrow \mathcal{M}(L)\rightarrow \mathcal{M}(L/N)\\
\rightarrow L/[N,L] \rightarrow L/L^2\rightarrow L/(L^2+N)\rightarrow 0,\ \ \ \ \ \ \
\ \ \ \ \ \nonumber
\end{eqnarray*}
where  $H_3(-)$ denotes  the third homology of a Lie algebra (see \cite{s-s-e}).  In fact, this is similar to the definition of the Schur multiplier of a pair of groups given by Ellis \cite{ge}.
Also it is easy to  see that if the ideal $N$ admits  a complement in $L$, then
$\mathcal{M}(L)\cong \mathcal{M}(N,L)\oplus \mathcal{M}(L/N)$.
So in this paper, we consider pairs $(N,L)$ of Lie superalgebras in which $N$ possesses a complement in $L$.

Now, let $0\to R\to F\to L\to 0$ be a  free presentation of a Lie superalgebra $L$ and $N\cong S/R$ such that $S$ is a graded ideal of $F$.
 We define the  Schur multiplier of the pair $(N,L)$ as
\[\mathcal{M}(N,L)=\dfrac{R\cap [S,F]}{[R,F]},\]
which is an abelian Lie superalgebra,  independent of the choice of the
free presentation of $L$.
Clearly if $N=L$, then $\mathcal{M}(L,L)=\mathcal{M}(L)$ is
 the  Schur multiplier of a Lie superalgebra  $L$, given in \cite{miao,n1}. In fact, there is an isomorphism of supermodules
$\mathcal{M}(L)\cong H_2(L)$, where $H_2(L)$ is the second homology of $L$ (see \cite[Corollary 6.5]{lad}).
 This notion also extends the Schur multiplier of a pair of Lie algebras given in \cite{s-s-e,safa,safa2}.
 Moreover,  the Schur multiplier (resp. the Schur $\mathsf{Lie}$-multiplier) of a pair of Leibniz algebras is studied in \cite{b-s,h-e-s}
 (resp. \cite{safa1}).\\

Let $L$ and $M$ be two Lie superalgebras. By an action of $L$ on $M$, we mean an $\mathbb{F}$-bilinear map
$L\times M\to M$ given by $(l,m)\mapsto\ ^lm$ satisfying
\begin{itemize}
\item[(i)] $^lm\in M_{a+b}$, for every $l\in L_a$ and $m\in M_b$, $a,b\in\mathbb{Z}_2$ (even grading),
\item[(ii)] $^{[l,l']}m =\ ^l(^{l'}m)-(-1)^{|l||l'|}\ ^{l'}(^lm)$,
\item[(iii)] $^l[m,m']=[^lm,m']+(-1)^{|l||m|}[m, ^lm']$,
\end{itemize}
for all $l,l'\in L$ and $m,m'\in M$.
Clearly, if $L$ is a sub-superalgebra of some Lie superalgebra $P$ and $M$ is a graded ideal of $P$, then the Lie multiplication of $P$ induces an action of $L$ on $M$ by $^lm=[l,m]$. We also define the semidirect product of $M$ by $L$, denoted by
$M\rtimes L$, with underlying supermodule $M\oplus L$ and  the graded bracket given by
$$[(m,l),(m',l')]=\big{(}[m,m']+\ ^lm' -(-1)^{|m||l'|}\ ^{l'}m, [l,l']\big{)},$$
for $m,m'\in M$ and $l,l'\in L$ (see \cite{lad} for more information).\\

Now we define the concept of a cover for a pair of Lie superalgebras. The usual notion of central extensions and covers of a Lie superalgebra
is already given in  \cite{n1,miao}.
\begin{definition}\label{def1} \normalfont
Let $(N,L)$ be a pair of Lie superalgebras.  A {\it relative central extension} of  $(N,L)$  is a  homomorphism of Lie superalgebras $\sigma: M\to L$ together with an action of $L$ on $M$ satisfying  the following conditions:
\begin{itemize}
\item[(i)] $\sigma(M)=N$,
\item[(ii)] $\sigma(^lm)=[l,\sigma(m)]$, for all $l\in L$, $m\in M$,
\item[(iii)] $^{\sigma(m')}m=[m',m]$, for all $m,m'\in M$,
\item[(iv)]  $\ker \sigma\subseteq Z(M,L)$,
in which $Z(M,L)=\{m\in M|\ ^lm=0, \forall l\in L\}$.
\end{itemize}
In addition, the relative central extension $\sigma:M\rightarrow L$ is said to be a {\it cover} of  $(N,L)$, if
 $\mathcal{M}(N,L)\cong\ker \sigma\subseteq [M,L],$
 where $[M,L]=\langle ^lm\mid  m\in M, l\in L\rangle.$
\end{definition}

\begin{proposition}\label{prop21}
Every pair  of  Lie superalgebras has at least one cover.
\end{proposition}
\begin{proof}
Let $(N,L)$ be a pair of Lie superalgebras and $0\to R\to F\stackrel{\pi}{\rightarrow}L\to 0$ be a free presentation of $L$ such that $N\cong S/R$, for some
graded ideal $S$ of $F$. Consider $T/[R,F]$ as a complement of $\mathcal{M}(N,L)$ in the abelian Lie superalgebra $R/[R,F]$. It is easy to see
that   $^{l}{(s+T)}= [f,s]+T$ where $\pi(f)=l$, is an action of $L$ on $S/T$.   Now, define
 $\sigma:S/T\to L$ by $s+T\mapsto \pi(s)$. We show that $\sigma$ is a cover of $(N,L)$.
 Clearly $\sigma(M)=N$.
 Also, $\sigma(^{l}{(s+T)})= \sigma([f,s]+T)=\pi([f,s])=[l,\pi(s)]=[l,\sigma(s+T)]$ and
 $^{\sigma(s'+T)}(s+T)=\ ^{\pi(s')}(s+T)=[s'+T,s+T]$. Moreover, one may easily check that $\mathcal{M}(N,L)\cong\ker\sigma\subseteq Z(M,L)$.
 Finally
 \begin{eqnarray*}
\ker\sigma&=&\frac{R}{T}\cong\frac{R\cap[S,F]}{[R,F]}\subseteq\frac{[S,F]}{[R,F]}=\frac{[S,F]}{[S,F]\cap T}\cong\frac{[S,F]+T}{T}\\
&=&\langle [f,s]+T|\ f\in F, s\in S\rangle=\langle ^{\pi(f)}{(s+T)}|\ f\in F, s\in S \rangle=[S/T,L].
\end{eqnarray*}
This completes the proof.
\end{proof}


\section{Main results}

The Lie algebra analogue of Schur's theorem \cite{s},  proved by
Moneyhun \cite{m}, states that if L is a Lie algebra such that $\dim(L/Z(L))=m$, then $\dim L^2\leq\frac{1}{2}m(m-1)$.
This upper bound has been recently generalized to Lie superalgebras. In \cite{n1}, Nayak shows that if $L$ is a Lie superalgebra with $\dim(L/Z(L))=(m|n)$,
then $\dim L^2\leq \frac{1}{2}\big{(}(m+n)^2+(n-m)\big{)}$. In the following theorem, we  extend this result to a pair of Lie superalgebras.

\begin{theorem}\label{th1}
Let $(N,L)$ be a pair of Lie superalgebras such that $L/N$ is finite dimensional and  $\dim(N/Z(N,L))=(m|n)$. Then
\[\dim [N,L]\leq \frac{1}{2}\big{(}(m+n)^2+(n-m)\big{)}+(m+n)\dim(L/N).\]
\end{theorem}
\begin{proof}
Let $\{\bar{x}_1,\ldots,\bar{x}_m,\bar{y}_{1},\ldots,\bar{y}_{n}\}$ be a basis of $\bar{N}=N/Z(N,L)$ in which
$\bar{x}_i\in\bar{N}_0$ ($1\leq i\leq m$) and $\bar{y}_{i}\in \bar{N}_1$ ($1\leq i\leq n$). Also let $\dim(L/N)=(p|q)$.  Since $\dim (L/Z(N,L))=(m+p|n+q)$, one may extend the above basis to
 $$\big{\{}\bar{x}_1,\ldots,\bar{x}_m,\bar{x}_{m+1},\ldots,\bar{x}_{m+p},\bar{y}_1,\ldots,\bar{y}_n,\bar{y}_{n+1},\ldots,\bar{y}_{n+q}\big{\}}$$
  for $\bar{L}=L/Z(N,L)$ where
$\bar{x}_j\in\bar{L}_0$ ($1\leq j\leq m+p$) and $\bar{y}_j\in \bar{L}_1$ ($1\leq j\leq n+q$). Then
\[\dim[N_0,L_0]\leq \# \Big{\{}[x_i,x_j]: 1\leq i\leq m\ and\ i<j\leq m+p \Big{\}}={m\choose 2}+mp,\]
\[\dim[N_1,L_1]\leq \#\Big{\{}[y_i,y_j]: 1\leq i\leq n\ and\ i\leq j\leq n+q \Big{\}}={n\choose 2}+n+nq,\]
\[\dim\Big{(}[N_0,N_1]+[N_0,L_1 - N_1]+[N_1,L_0- N_0]\Big{)}\leq mn+mq+np.\]
Therefore,
\begin{eqnarray*}
\dim [N,L]&\leq& {m\choose 2}+mp+{n\choose 2}+n+nq+mn+mq+np\\
&=& \frac{1}{2}\big{(}(m+n)^2+(n-m)\big{)}+(m+n)(p+q).\\
\end{eqnarray*}
\end{proof}

\begin{remark}\label{rem2}\normalfont
Clearly if $N=L$, then our upper bound coincides with the Nayak's one. Also if $(N,L)$ is a pair of Lie algebras, i.e.  $\dim(N/Z(N,L))=(m|0)$, then
the above theorem implies that $\dim [N,L]\leq \frac{1}{2}m\big{(}m+2\dim(L/N)-1\big{)}$ which is  given in \cite{s-s-e}.
\end{remark}

In the next result, we provide an upper bound on the dimension of the Schur multiplier of a pair of Lie superalgebras.
\begin{corollary}\label{coro3}
Let $(N,L)$ be a pair of finite dimensional Lie superalgebras such that  $\dim N=(m|n)$. Then
\[\dim \mathcal{M}(N,L)\leq \frac{1}{2}\big{(}(m+n)^2+(n-m)\big{)}+\dim N\dim(L/N)-\dim[N,L].\]
\end{corollary}
\begin{proof}
Using the notation of Proposition \ref{prop21}, we have
\[[N,L]\cong\frac{[S,F]+R}{R}\cong\frac{[S,F]/[R,F]}{(R\cap [S,F])/[R,F]},\]
and hence
\begin{equation}\label{eq1}
\dim [N,L]+\dim\frac{R\cap [S,F]}{[R,F]}=\dim\frac{[S,F]}{[R,F]}.
\end{equation}
Also, since
\[\dim\frac{S/[R,F]}{Z\big{(}S/[R,F], F/[R,F]\big{)}}\leq \dim\frac{S/[R,F]}{R/[R,F]}=(m|n)\]
and $\frac{F/[R,F]}{S/[R,F]}\cong L/N$,
by Theorem \ref{th1} we get
\[\dim\frac{[S,F]}{[R,F]}\leq \frac{1}{2}\big{(}(m+n)^2+(n-m)\big{)}+(m+n)\dim(L/N).\]
Now equality (\ref{eq1}) completes the proof.
\end{proof}

\begin{remark}\label{rem4} \normalfont
One interesting problem in the theory of Lie algebras is the characterization of  Lie algebras using  $t(L)=\frac{1}{2}m(m-1)-\dim\mathcal{M}(L)$,
in which $L$ is an $m$-dimensional Lie algebra. The characterization  of nilpotent Lie algebras for $0\leq t(L)\leq 8$ has been studied in
\cite{b-m-s,h,h-s}. Similarly, Green's result \cite{grn}  yielded a lot of interest on the classification of finite $p$-groups by $t(G),$ investigated by several authors (see \cite{brk, lls, w,zh}).

Now, let $L$ be a Lie superalgebra.  A natural  question arises whether one could characterize  $(m|n)$-dimensional
 Lie superalgebras by
$t(L)=\frac{1}{2}\big{(}(m+n)^2+(n-m)\big{)}-\dim\mathcal{M}(L)$.
Lemma \ref{lem31} below is the answer for $t(L)=0$,
and in \cite{miao} this question is  answered  for $t(L)\leq2$. In \cite{safa3}, this is discussed for $n$-Lie superalgebras.
\end{remark}

In what follows, we deal with a similar problem   for the pair case.

\begin{lemma}\cite{m}\label{lem3}
If $L$ is an $m$-dimensional  Lie algebra, then
$\dim \mathcal{M}(L)\leq\frac{1}{2}m(m-1)$. Moreover the equality holds if and only if $L$ is abelian.
\end{lemma}

\begin{lemma}\cite{n1}\label{lem31}
If $L$ is a  Lie superalgebra of dimension $(m|n)$, then
$$\dim \mathcal{M}(L)\leq \frac{1}{2}\big{(}(m+n)^2+(n-m)\big{)}.$$ Moreover the equality holds if and only if $L$ is abelian.
\end{lemma}

\begin{theorem}\label{th2}
Let $(N,L)$ be a pair of finite dimensional Lie superalgebras such that  $\dim N=(m|n)$. Then
\[\dim \mathcal{M}(N,L)\leq \frac{1}{2}\big{(}(m+n)^2+(n-m)\big{)}+\dim N \dim(L/N).\]
In particular, if $L$ is abelian then the equality holds, and if the equality holds then $N\subseteq Z(L)$.
\end{theorem}
\begin{proof}
 Corollary \ref{coro3} yields  the above inequality. Clearly, if the equality holds, then Corollary \ref{coro3} implies that $N$ is central.
Now, suppose that $L$ is abelian and put $\dim (L/N)=(p|q)$. Then by Lemma \ref{lem31} we have
$\dim\mathcal{M}(L/N)=\frac{1}{2}\big{(}(p+q)^2+(q-p)\big{)}$ and
 $\dim\mathcal{M}(L)=\frac{1}{2}\big{(}(m+p+n+q)^2+(n+q-m-p)\big{)}$. Thus the isomorphism
 $\mathcal{M}(L)\cong\mathcal{M}(N,L)\oplus\mathcal{M}(L/N)$ implies that
 \begin{equation*}\label{eq0}
\dim \mathcal{M}(N,L)= \frac{1}{2}\big{(}(m+n)^2+(n-m)\big{)}+(m+n)\dim(L/N),
\end{equation*}
which completes the proof.
\end{proof}

Recall from \cite{b-l,b-m-s} that a finite dimensional Lie (super)algebra $L$ is called Heisenberg, if $L^2=Z(L)$ and $\dim L^2=1$.
A Heisenberg Lie algebra, denoted by $H(m)$, has dimension $2m+1$ with a basis $\{ x_1,\ldots,x_{2m},z\}$ and non-zero multiplications
$[x_i,x_{m+i}]=z$, for $1\leq i\leq m$.
 Heisenberg superalgebras consist of two types according to the
parity of the center.

(1)  Heisenberg superalgebra of even center and dimension $(2m+1|n)$:
\[\mathcal{H}(m,n)=\langle x_1,\ldots,x_{2m},z\rangle\oplus\langle  y_1,\ldots, y_n\rangle \ \ \ \ (m+n\geq 1)\]
with non-zero multiplications
$[x_i,x_{m+i}]=z=[y_j,y_j]$, for $1\leq i\leq m$ and  $1\leq j\leq n$.

(2) Heisenberg superalgebra of odd center and dimension $(n|n+1)$:
\[\mathcal{H}(n)=\langle x_1,\ldots,x_n\rangle\oplus\langle y_1,\ldots,y_n,z\rangle \ \ \ \ (n\geq 1).\]
with non-zero multiplications $[x_i,y_i]=z$, for $1\leq i\leq n$.\\

The following lemmas are needed for proving the next results.

\begin{lemma}\cite[Theorem 3]{b-m-s}\label{lem4}
Let $L$ be an $m$-dimensional nilpotent Lie algebra. Then
$\dim \mathcal{M}(L)=\frac{1}{2}m(m-1)-1$ if and only if $L\cong H(1)$.
\end{lemma}

\begin{lemma}\cite[Theorem 4.3]{n1}\label{lem5}
For $\mathcal{H}(m,n)$ we have
\begin{eqnarray*}
\dim\mathcal{M}(\mathcal{H}(m,n))=\left\{\begin{array}{ll}
2m^2 -m-1+2mn+\frac{1}{2}n(n+1) &if \ m+n\geq2,\\
0 &if \ m=0, n=1,\\
2 &if \ m=1, n=0.
\end{array} \right.
\end{eqnarray*}
\end{lemma}

\begin{lemma}\cite[Proposition 4.5]{miao}\label{lem51}
For $\mathcal{H}(n)$ we have
\begin{eqnarray*}
\dim\mathcal{M}(\mathcal{H}(n))=\left\{\begin{array}{ll}
2n^2 -1 &if \ n\geq2,\\
2 &if \  n=1.
\end{array} \right.
\end{eqnarray*}
\end{lemma}

Now, we are ready to prove the following main theorem.
\begin{theorem}\label{th5}
Let $(N,L)$ be a pair of finite dimensional nilpotent Lie superalgebras with  $\dim N=(m|n)$. If
\begin{equation}\label{eq2}
\dim \mathcal{M}(N,L)= \frac{1}{2}\big{(}(m+n)^2+(n-m)\big{)}+\dim N\dim(L/N)-1,
\end{equation}
then $L$ is a non-abelian Lie superalgebra in which one of the following holds:
\begin{itemize}
\item[$(i)$] $[N,L]=0$,
\item[$(ii)$] $\dim [N,L]=1$  and $[N,L]=Z(N,L)$.
\end{itemize}
\end{theorem}
\begin{proof}
Clearly Theorem \ref{th2} and  equality (\ref{eq2})  imply that $L$ is non-abelian.
Also Corollary \ref{coro3} and  equality (\ref{eq2})  imply that $\dim[N,L]\leq 1$. Suppose that $[N,L]\not=0$.
Since $L$ is nilpotent, we have $\dim [[N,L],L]\lneqq\dim[N,L]=1$ and hence $[N,L]\subseteq Z(N,L)$.
 We only need to show that $\dim Z(N,L)=1$.

Considering Proposition \ref{prop21}, let $\sigma:M\to L$ be a cover of the pair $(N,L)$.
Put $P=M\rtimes \frac{L}{N}$, in which the action of $L/N$ on $M$ is induced by the one of $L$ on $M$.
By construction we have $P/M\cong L/N$, $Z(M,P)=Z(M,L)$, $[M,L]= [M,P]$, $P/Z(M,P)\cong L$ and $M/Z(M,P)\cong N$.
Therefore,
\[\dim\frac{M}{Z(M,P)}=\dim\frac{M}{Z(M,L)}\leq \dim\frac{M}{\ker\sigma}=\dim N=(m|n),\]
and  by applying  Theorem \ref{th1} on the pair $(M,P)$ we get
\begin{eqnarray}\label{eq21}
\dim\mathcal{M}(N,L)&=&\dim \ker\sigma \leq \dim [M,L]=\dim [M,P]\nonumber\\
&\leq& \frac{1}{2}\big{(}(m+n)^2+(n-m)\big{)}+(m+n)\dim(L/N).
\end{eqnarray}
One may consider the following two cases:\\

{\bf Case I.}  $\ker\sigma \subsetneqq Z(M,L)$.\\
Then $\dim (M/Z(M,P))\lneqq \dim(M/\ker\sigma)=(m|n)$, and we have two cases again:

{\it Subcase I-1.} If $\dim (M/Z(M,P))\leq(m|n-1)$, then Theorem \ref{th1} and equality (\ref{eq2}) imply that
\begin{eqnarray*}
&&\frac{1}{2}\big{(}(m+n)^2+(n-m)\big{)}+(m+n)\dim(L/N)-1 \leq \dim [M,P]\\
&&\leq \frac{1}{2}\big{(}(m+n-1)^2+(n-1-m)\big{)}+(m+n-1)\dim(L/N).
\end{eqnarray*}
This  yields  $\dim L\leq 1$, and in both cases $\dim L=(1|0)$ and $(0|1)$, $L$ is abelian which is a contradiction.

{\it Subcase I-2.} If $\dim (M/Z(M,P))\leq(m-1|n)$, then by a similar manner one can show that $\dim L\leq 2$.
If $\dim L=(2|0)$, then $L$ is a Lie algebra. Since it is non-abelian, by Lemma \ref{lem3}, $\dim\mathcal{M}(L)=0$.
Hence Lemma \ref{lem4} implies that $\dim L=3$, which is a contradiction.
If $\dim L=(0|2)$, then $L$ is abelian which is impossible. Thus $\dim L =(1|1)$. Let $L=L_0\oplus L_1$ where
$L_0=\langle x\rangle$ and $L_1=\langle y\rangle$.
Then Definition \ref{def0} implies that there exist only two non-abelian Lie superalgebras. The  only  non-zero multiplication of the first one
is $[x,y]=y$, but this Lie superalgebra is not nilpotent since $Z(L)=0$. In the second one, we must have $[y,y]=x$ as the only
 non-zero multiplication. In this case $L^2=Z(L)=\langle x\rangle$ and   $L\cong \mathcal{H}(0,1)$.
Thus by Lemma \ref{lem5}, we get $\dim\mathcal{M}(L)=0$. Now, if $\dim N=(1|0)$, then
 $N=Z(L)$ which is impossible since $[N,L]\neq 0$. Also the multiplication $[y,y]=x$ shows that $\dim N$ cannot be $(0|1)$. Hence  we must have $N=L$. But equality (\ref{eq2}) implies that
 $\dim \mathcal{M}(L)=\dim \mathcal{M}(L,L)=1$, which is impossible.
 Hence Case I leads to a contradiction.\\

{\bf Case II.}  $\ker\sigma= Z(M,L)$.\\
 By  (\ref{eq2}) and  (\ref{eq21}), we have  two cases:

 {\it Subcase II-1.}
 $\ker\sigma=Z(M,P)=[M,P]$. Then we have  $N\cong \frac{M}{Z(M,P)}\subseteq Z(\frac{P}{Z(M,P)})\cong Z(L)$, which is impossible.

 {\it Subcase II-2.} $\ker\sigma=Z(M,P)\subsetneqq[M,P]$ and hence
 \begin{eqnarray}\label{eq22}
\dim [M,P]= \frac{1}{2}\big{(}(m+n)^2+(n-m)\big{)}+\dim N\dim(L/N).
\end{eqnarray}
In this subcase, we show that $\dim Z(\frac{M}{Z(M,P)},\frac{P}{Z(M,P)})=\dim Z(N,L)=1$. As mentioned above, this completes the proof.
 Suppose on the contrary that $Z(\frac{M}{Z(M,P)},\frac{P}{Z(M,P)})$ has a basis containing at least two elements $x+Z(M,P)$ and $y+Z(M,P)$.
Considering the adjoint map $ad_x:P\to P^2$, we have $P/C_P(x)\cong [x,P]$, where $C_P(x)$ is the centralizer of $x$ in $P$.
Since $x+Z(M,P)$ is non-zero, $Z(M,P)\subsetneqq C_P(x)$ and hence
 \begin{eqnarray}\label{eq23}
\dim [x,P]=\dim (P/C_P(x))\lneqq\dim(P/Z(M,P))=\dim L.
\end{eqnarray}
We prove that $\dim [x,P]=\dim L-1$.
Since $x+[x,P]\in Z(\frac{M}{[x,P]},\frac{P}{[x,P]})- \frac{Z(M,P)}{[x,P]}$, we have
\[\dim\Big{(}\frac{P/[x,P]}{Z(\frac{M}{[x,P]},\frac{P}{[x,P]})}\Big{)}\lneqq\dim\Big{(}\frac{P/[x,P]}{\frac{Z(M,P)}{[x,P]}}\Big{)}=\dim L,\]
and since
\[\dim\Big{(}\frac{P/[x,P]}{Z(\frac{M}{[x,P]},\frac{P}{[x,P]})}\Big{)}-\dim\Big{(}\frac{M/[x,P]}{Z(\frac{M}{[x,P]},\frac{P}{[x,P]})}\Big{)}=\dim (P/M)=\dim (L/N),\]
we get
\[\dim\Big{(}\frac{M/[x,P]}{Z(\frac{M}{[x,P]},\frac{P}{[x,P]})}\Big{)}\lneqq\dim L -\dim (L/N)=\dim N=(m|n).\]
Now if $\dim\Big{(}\frac{M/[x,P]}{Z(\frac{M}{[x,P]},\frac{P}{[x,P]})}\Big{)}\leq (m|n-1)$, then by applying Theorem \ref{th1} for the pair
$(\frac{M}{[x,P]},\frac{P}{[x,P]})$, and also using equality (\ref{eq22})
and a similar manner to subcase I-1, we get $\dim [x,P]\geq \dim L$, which contradicts  (\ref{eq23}).\\
But if  $\dim\Big{(}\frac{M/[x,P]}{Z(\frac{M}{[x,P]},\frac{P}{[x,P]})}\Big{)}\leq (m-1|n)$ then by a similar computation
we get $\dim [x,P]\geq \dim L-1$,
and using  (\ref{eq23}) we have
$\dim [x,P]=\dim(\frac{P}{C_P(x)})=\dim L-1$. Therefore,
 \begin{eqnarray}\label{eq24}
\dim\big{(}\frac{C_P(x)}{Z(M,P)}\big{)}&=&\dim\big{(}\frac{P}{Z(M,P)}\big{)}-\dim\big{(}\frac{P}{C_P(x)}\big{)}\nonumber\\
&=&\dim L-(\dim L-1)=1.
\end{eqnarray}
Since $x+Z(M,P)\in Z(\frac{M}{Z(M,P)},\frac{P}{Z(M,P)})$, then  $[M,P]\subseteq C_P(x)$, and since in this subcase we have $Z(M,P)\subsetneqq [M,P]$,
thus
 \begin{eqnarray*}
\dim L-1&=&\dim\big{(}\frac{P}{C_P(x)}\big{)}\leq \dim\big{(}\frac{P}{[M,P]}\big{)}\\
&=& \dim\Big{(}\frac{P/Z(M,P)}{[\frac{M}{Z(M,P)},\frac{P}{Z(M,P)}]}\Big{)}=\dim\big{(}\frac{L}{[N,L]}\big{)}\\
&=&\dim L -1,
\end{eqnarray*}
which implies that $[M,P]= C_P(x)$. Since $x+Z(M,P)$ is arbitrary, we also have $[M,P]= C_P(y)$ and hence $y\in C_P(x)$.
But this contradicts  (\ref{eq24}). Therefore in this subcase, we must have $\dim Z(\frac{M}{Z(M,P)},\frac{P}{Z(M,P)})=\dim Z(N,L)=1$.
\end{proof}

The following corollary generalizes
Lemma \ref{lem4} and also  shows that there is no finite dimensional  nilpotent Lie superalgebra with non-trivial odd part
satisfying equality (\ref{eq25}) below (see also \cite[Proposition 4.8]{miao}).

\begin{corollary}\label{coro6}
Let $L$ be a nilpotent Lie superalgebra of dimension $(m|n)$. Then
\begin{eqnarray}\label{eq25}
\dim \mathcal{M}(L)= \frac{1}{2}\big{(}(m+n)^2+(n-m)\big{)}-1
\end{eqnarray}
 if and only if $L\cong\mathcal{H}(1,0)=H(1)$.
\end{corollary}
\begin{proof}
If $L\cong \mathcal{H}(1,0)$,  then Lemma \ref{lem5}  implies that
 $\dim\mathcal{M}(\mathcal{H}(1,0))=2$. Conversely,
if equality (\ref{eq25}) holds, then  Theorem \ref{th5}  implies  that $L^2=Z(L)$ is one-dimensional.
If $\dim Z(L)=(1|0)$, then $L$ is the  Heisenberg Lie superalgebra $\mathcal{H}(\frac{m-1}{2},n)$ with even center.
Using  Lemma \ref{lem5},
we must have $L\cong \mathcal{H}(1,0)$. If $\dim Z(L)=(0|1)$, then
$L$ is the  Heisenberg Lie superalgebra $\mathcal{H}(\frac{m+n-1}{2})$ with odd center.  In this case, equality (\ref{eq25}) and
  Lemma \ref{lem51} lead to a contradiction.
\end{proof}


Hesam Safa \\
Department of Mathematics, Faculty of Basic Sciences, University of Bojnord, Bojnord, Iran.\\
E-mail address: hesam.safa@gmail.com, \ \ \  h.safa@ub.ac.ir \\

\end{document}